\documentclass[a4paper,12pt,oneside]{amsart}
\usepackage[margin=3.25cm]{geometry}

\theoremstyle{plain}

\newtheorem{theorem}{Theorem}[section]
\newtheorem{proposition}[theorem]{Proposition}
\newtheorem{lemma}[theorem]{Lemma}
\newtheorem{corollary}[theorem]{Corollary}
\newtheorem{remark}[theorem]{Remark}

\usepackage{soul}
\usepackage{graphicx}
\usepackage{hyperref}

\newtheorem{definition}[theorem]{Definition}

\begin{document}

\title[Cartan ribbonization]{Cartan ribbonization\\and a topological inspection}

\author{Matteo Raffaelli}
\address{DTU Compute\\
Technical University of Denmark\\
2800 Kongens Lyngby\\
Denmark}
\email{matraf@dtu.dk}

\author{Jakob Bohr}
\address{DTU Nanotech\\
Technical University of Denmark\\
2800 Kongens Lyngby\\
Denmark}
\email{jabo@nanotech.dtu.dk}

\author{Steen Markvorsen}
\address{DTU Compute\\
Technical University of Denmark\\
2800 Kongens Lyngby\\
Denmark}
\email{stema@dtu.dk}

\date{September 30, 2018}

\begin{abstract}
We develop the concept of Cartan ribbons together with a rolling-based method to ribbonize and approximate any given surface in space by intrinsically flat ribbons. The rolling requires that the geodesic curvature along the contact curve on the surface agrees with the geodesic curvature of the corresponding Cartan development curve. Essentially, this follows from the orientational alignment of the two co-moving Darboux frames during rolling. Using closed contact center curves we obtain closed approximating Cartan ribbons that contribute zero to the total curvature integral of the ribbonization. This paves the way for a particularly simple topological inspection -- it is reduced to the question of how the ribbons organise their edges relative to each other. The Gauss--Bonnet theorem leads to this topological inspection of the vertices.
Finally, we display two examples of ribbonizations of surfaces, namely of a torus using two ribbons, and of an ellipsoid using closed curvature lines as center curves for the ribbons.
\end{abstract}

\maketitle

\section{Introduction}

The approximation of surfaces by patch-works of planar parts has a long use in fundamental and applied mathematics. Foremost comes to mind the multifaceted applications of triangulations \cite{schumaker1993}. In the present work we develop a scheme for approximating a surface by the use of multiple developable surfaces. Some of the beauty of this approach is the relatively few numbers of developable stretches -- ribbons -- needed to approximate a given  surface. Not to mention that the study of shapes and structures of developable surfaces is itself a classical subject that has intrigued mathematicians for centuries and has found numerous artistic applications in architecture and design, see  \cite{lawrence2011}.

In the seventies K. Nomizu pointed out that the concept of (extrinsic) rolling can be understood as a kinematic interpretation of the (intrinsic) Levi-Civita connection and of the Cartan development of curves, see \cite{nomizu1978} and \cite{Kobayashi1963}. One derives simple expressions for the components of the corresponding relative angular velocity vector of the rolling, i.e. the geodesic torsion, the normal curvature, and the geodesic curvature of the given curve and its development, see \cite{tuncer2007,cui2010,molina2014,hananoi2017, izumiya2015}. For example, in conjunction with a plane, the rolling must propagate along a planar curve which has the same geodesic curvature as the given curve, see examples in \cite{raffaelli2016}.

In recent years rolling has received a renewed wave of interest -- in part because of its importance for robotic manipulation of objects \cite{cui2015}. For example, there has been an interest in understanding rolling from symmetry arguments \cite{chitour2015} as well as purely geometrical considerations \cite{tuncer2010,chitour2014,krakowski2016}. Also, the shapes known as D-forms are examples of surface structures that are formed by assembling several developable surfaces \cite{sharp2005,wills2006,orduno2016}.

The paper is organized as follows: In sections \ref{secInitial} and \ref{secDevCartanRib} we apply the notion of rolling as an alternative entrance to the construction of developable surface approximations. We show how the method of rolling a surface along the planar Cartan development of a given curve on the surface produces a planar ribbon which -- after isometric bending along the lines of the instantaneous rotation axes -- will reproduce the surface approximation along the said curve. In other words, the rolling induces a local isometry between the flat approximation along the curve and the plane. Further in section \ref{secDevCartanRib} we discuss a specific measure of the  local goodness of a given ribbon approximation.
In section \ref{secGBinspect} we then initiate the corresponding study of such approximations by establishing a precise calculation of the Euler characteristic of the surfaces via an inspection of the family of approximating ribbons. Finally, in sections \ref{secExampTorus} and  \ref{secExampEllipsoid}, we illustrate the approximation method by two concrete examples which show the ensuing Cartan ribbon approximations of a torus (along two trigonometric center curves) and of an ellipsoid (along six lines of curvature), respectively.

\section{The Initial Setting} \label{secInitial}

We consider two surfaces $S$ and $\tilde{S}$ in $\mathbb{R}^{3}$.
Let $\gamma$ be a smooth, regular
curve on $S$, $ \gamma: \, J=[0,\alpha]\to S$, such that $\gamma(0)=(0,0,0)$. We equip $\gamma$ with its \emph{Darboux frame field} $\mathcal{F} = \{e, h, N\}$, defined as follows: for each $t \in J$ we let $N(t)$ denote a unit normal vector to $S$ at $\gamma(t)$, we let  $e(t) = \gamma'(t)/\Vert \gamma'(t) \Vert$ the unit tangent vector of $\gamma$ and $h(t) = N(t)\times e(t)$. The frame $\mathcal{F}$ then satisfies the following equations -- see for example  \cite[Corollary 17.24]{Gray}:
\begin{equation} \label{eq:01}
\begin{bmatrix}
e'(t) \\
h'(t) \\
N'(t)
\end{bmatrix}
= \Vert \gamma'(t) \Vert \cdot
\begin{bmatrix}
0                      &  \kappa_{g}(t)  &  \kappa_{n}(t) \\
-\kappa_{g}(t)  &  0                     &  \tau_{g}(t)      \\
-\kappa_{n}(t)  &  -\tau_{g}(t)      &  0
\end{bmatrix}
\begin{bmatrix}
e(t) \\
h(t) \\
N(t)
\end{bmatrix},
\end{equation}
where $\tau_{g}(t)$, $\kappa_{n}(t)$, and $\kappa_{g}(t)$ are the geodesic torsion, the normal curvature, and the geodesic curvature, respectively, of $\gamma$ at $\gamma(t)$. Since we are so far only considering local geometric entities, the surfaces $S$ and $\tilde{S}$ need not be orientable, i.e. the frame $\mathcal{F}$ and its properties -- such as the signs appearing in \eqref{eq:01} -- depend on the local choice of normal vector field $N$. In the final sections we will note a few consequences concerning the rolling and the corresponding ribbonization of non-orientable surfaces.

\subsection{Moving $S$ on $\tilde{S}$} \label{secSonStilde}

Given a curve $\gamma$ on $S$ as above, we now consider smooth and regular curves $\tilde{\gamma}$ on the other surface $\tilde{S}$ such that the  following initial compatibility and contact conditions are satisfied:
\begin{equation}
\begin{aligned}
\tilde{\gamma}(0)&=\gamma(0)  = (0,0,0)\\
\tilde{\gamma}'(0)&=\gamma'(0)\\
\Vert \tilde{\gamma}^{\prime}\Vert &= \Vert \gamma^{\prime} \Vert,
\end{aligned}
\end{equation}
so that $\tilde{\gamma}$ has the same initial point and direction as $\gamma$ and so that $\tilde{\gamma}$ has the same speed as $\gamma$ for all $t \in J$. A framed motion of $(S, \gamma)$ on $\tilde{S}$ is then defined as follows:

\begin{definition} \label{defMotion}
Let $E^{+}(3)$ be the group of direct isometries of $\mathbb{R}^3$. A (1-\emph{parameter}) \emph{framed motion} $g_{t}$ \emph{of} $(S, \gamma)$ \emph{on} $\tilde{S}$ \emph{along} $\tilde{\gamma}$ is a differentiable map $J \to E^{+}(3)$ such that for each $t$ the map $g_{t}$ is the isometry that maps
\begin{equation}
\begin{aligned}
\gamma(t) &\text{ to } \tilde{\gamma}(t), \\
\gamma(t) + e(t) &\text{ to } \tilde{\gamma}(t) + \tilde{e}(t), \\
\gamma(t) + N(t) &\text{ to } \tilde{\gamma}(t) + \tilde{N}(t) ,
\end{aligned}
\end{equation}
where $\tilde{e}$ and  $\tilde{N}$ are two of the members of the Darboux frame $\tilde{\mathcal{F}} = \{\tilde{e}\, , \,\, \tilde{h} = \tilde{N}\times \tilde{e} \, , \, \, \tilde{N} \}$ along $\tilde{\gamma}$ on $\tilde{S}$ defined in the same way as the frame ${\mathcal{F}}$ along $\gamma$ on $S$.
The point $\tilde{\gamma}(t)$ is called the \emph{contact point} at instant $t$, and $\tilde{\gamma}(J)$ is called the \emph{contact curve} of the framed motion $g_{t}$ of $(S, \gamma)$ on $\tilde{S}$.
\end{definition}

Since $g_{t}$ is in particular an instantaneous isometry it is represented by $x \mapsto R_{t} x + c_{t}$, where $R_{t} \in SO(3)$ is a rotation matrix and $c_{t}$ a translation vector. The instantaneous framed motion is then given by the vector field $V_{t} \, \colon x \mapsto \Omega_{t} (x - c_{t}) + c'_{t}$, with $\Omega_{t} = R'_{t} R_{t}^{\mathsf{T}}$, see \cite{nomizu1978}.
As $g_{t}$ is a \emph{framed motion} we have:

\begin{proposition} \label{propD}
Let $D_{t}$ be the matrix having $e(t)$, $h(t)$ and $N(t)$ as coordinate column vectors (with respect to a fixed coordinate system in $\mathbb{R}^{3}$) and similarly, let  $\tilde{D}_{t}$ be the matrix having $\tilde{e}(t)$, $\tilde{h}(t)$ and $\tilde{N}(t)$ as coordinate column vectors (with respect to the same fixed coordinate system in $\mathbb{R}^{3}$). Then
\begin{equation}
\begin{aligned}
R_{t} &= \tilde{D}_{t} D_{t}^{\mathsf{T}} \\
c_{t} &= \tilde{\gamma}(t) - R_{t}\gamma(t) ,
\end{aligned}
\end{equation}
so that
\begin{equation} \label{eqgt}
g_{t}(x) = \tilde{D}_{t}D_{t}^{\mathsf{T}}(x - \gamma(t)) + \tilde{\gamma}(t)  .
\end{equation}
\begin{proof}
The rotation $\tilde{D}_{t}D_{t}^{\mathsf{T}}$ maps the vector $e(t)$ to $\tilde{e}(t)$, and $N(t)$ to $\tilde{N}(t)$. The representation $g_{t}(x) = R_{t}x + c_{t}$ is therefore given by \eqref{eqgt}.
\end{proof}
\end{proposition}

\subsection{Rolling $S$ on $\tilde{S}$}

A framed motion $g_t$ of $(S, \gamma)$ on $\tilde{S}$ along $\tilde{\gamma}$ is said to be \emph{rotational} if, for all $t \in J$, $\Omega_{t}$ is different from the zero matrix.
At each time instant we can then find a unique vector $\omega_{t} \neq 0$, the \emph{angular velocity vector}, such that $\omega_{t} \times x = \Omega_{t}x$ for all $x \in \mathbb{R}^3$.

Based on the orientation of the angular velocity  vector relative to the common tangent plane of $g_{t}(S)$ and $\tilde{S}$, we introduce the following terminology for the instantaneous motion -- which extends directly to the entire motion.

\begin{definition} \label{defSpinetc}
The instantaneous rotational framed motion $g_t$ is a \emph{pure spinning} if the angular velocity vector $\omega_{t}$ is orthogonal to the tangent plane $T_{\tilde{\gamma}(t)}\tilde{S}$, and a \emph{pure twisting} if $\omega_{t}$ is proportional to the tangent vector $\tilde{e}(t)$. Finally, the motion $g_t$ will be called a \emph{standard rolling} if $\omega_{t}$ does not contain a spinning component and is not a pure twisting, i.e.~a standard rolling of  $S$ on $\tilde{S}$ is characterized by the condition that there exist smooth functions $a$ and $b$ such that $\omega_{t}$ decomposes as follows for all $t$:
\begin{equation} \label{eqCondRolling}
\omega_{t} = a(t)\cdot \tilde{e}(t) + b(t)\cdot \tilde{h}(t) + 0 \cdot \tilde{N}(t) , \quad b(t) \neq 0 \quad .
\end{equation}
\end{definition}

It turns out that a standard rolling of a given surface $S$ on a \emph{plane} gives a kinematic approach towards the construction of approximating developable ribbons that is presented below in section \ref{secDevCartanRib}. To begin with, we observe the following result for the more general situation of rolling $S$ on a general surface $\tilde{S}$:

\begin{proposition} \label{prop:1}
With the setting introduced above, a framed motion $g_{t}$ of $(S, \gamma)$ on $\tilde{S}$ along $\tilde{\gamma}$ is a standard rolling if and only if the following conditions are satisfied for all $t \in J$:
\begin{equation}\label{eqStandardRoll}
\begin{aligned}
\kappa_{g}(t) &= \tilde{\kappa}_{g}(t)\\
\kappa_{n}(t) & \neq \tilde{\kappa}_{n}(t) ,
\end{aligned}
\end{equation}
where $\tilde{\kappa}_{g}$ and $\tilde{\kappa}_{n}$ denote the geodesic curvature and the normal curvature of $\tilde{\gamma}$, respectively.
\end{proposition}

\begin{proof}
As in proposition \ref{propD}, $R_{t} = \tilde{D}_{t} D_{t}^{\mathsf{T}}$ and $c_{t} = \tilde{\gamma}(t) - R_{t}\gamma(t)$. Then, $g_{t}(x) = R_{t}x + c_{t}$, and so we can find the instantaneous motion $V_{t}$ by computing $\Omega_{t} (x - c_{t}) + c'_{t}$. Since $c'_{t} = \tilde{\gamma}'(t) - R'_{t} \gamma(t) - R_{t} \gamma'(t) = - R'_{t} \gamma(t)$ for $R_{t}$ maps $\gamma'(t)$ to $\tilde{\gamma}'(t)$, we obtain
\begin{equation}
\begin{aligned}
V_{t}(x)  &=  \Omega_{t} (x - \tilde{\gamma}(t) + R_{t}\gamma(t)) - R'_{t} \gamma(t) \\
              &=  \Omega_{t}x - \Omega_{t}\tilde{\gamma}(t) + R'_{t}\gamma(t)- R'_{t} \gamma(t) \\
              &=  \Omega_{t}(x - \tilde{\gamma}(t)),
\end{aligned}
\end{equation}
where $\Omega_{t} = R'_{t} R_{t}^{\mathsf{T}} = \tilde{D}'_{t} \tilde{D}_{t} + \tilde{D}_{t} D_{t}^{\prime \mathsf{T}} D_{t} \tilde{D}_{t}^{\mathsf{T}}$. If now we let
\begin{equation}
\varLambda_{t}=\Vert \gamma'(t) \Vert
\begin{bmatrix}
0  &  \kappa_{g}(t)  &  \kappa_{n}(t) \\
-\kappa_{g}(t)  &  0  &  \tau_{g}(t) \\
-\kappa_{n}(t)  &  -\tau_{g}(t)  &  0
\end{bmatrix},
\end{equation}
\noindent we have -- from \eqref{eq:01} -- that
$\tilde{D}^{\prime}_{t} =
\tilde{D}_{t}\tilde{\varLambda}_{t}^{\mathsf{T}} = - \tilde{D}_{t}\tilde{\varLambda}_{t}$ ($\tilde{\varLambda}_{t}$ is skew symmetric) as well as $D_{t}^{\prime \mathsf{T}} D_{t} = \varLambda_{t}$. Hence, if $\varXi_{t} = \varLambda_{t} - \tilde{\varLambda}_{t}$, that is
\begin{equation}
\varXi_{t} =
\begin{bmatrix}
0                         &  \varXi_{t}^{1,2}  &  \varXi_{t}^{1,3} \\
-\varXi_{t}^{1,2}  &  0                        &  \varXi_{t}^{2,3} \\
-\varXi_{t}^{1,3}  &  -\varXi_{t}^{2,3}  &  0
\end{bmatrix},
\end{equation}
where
\begin{equation}
\begin{aligned}
\varXi_{t}^{1,2} &= \Vert \gamma'(t) \Vert \cdot (\kappa_{g}(t) - \tilde{\kappa}_{g}(t)) \\
\varXi_{t}^{1,3} &= \Vert \gamma'(t) \Vert \cdot (\kappa_{n}(t) - \tilde{\kappa}_{n}(t)) \\
\varXi_{t}^{2,3} &= \Vert \gamma'(t) \Vert \cdot (\tau_{g}(t) - \tilde{\tau}_{g}(t)),
\end{aligned}
\end{equation}
the expression for $\Omega_{t}$ reduces to
\begin{equation} \label{eq:03}
\Omega_{t} = \tilde{D}_{t} \varXi_{t} \tilde{D}_{t}^{\mathsf{T}},
\end{equation}
and the resulting angular velocity vector of the rolling is thence -- with respect to the Darboux frame $\tilde{\mathcal{F}}(t) = \{\tilde{e}(t), \tilde{h}(t), \tilde{N}(t)\}$ along $\tilde{\gamma}$ in $\tilde{S}$:

{\small
\begin{equation} \label{eqomegaTilde}
\begin{aligned}
\omega_{t} &=\left(-\varXi_{t}^{2,3}\, , \, \, \varXi_{t}^{1,3} \, , \, \,  -\varXi_{t}^{1,2}   \right)_{\tilde{\mathcal{F}}(t)}  \\
&= \Vert \gamma'(t) \Vert \cdot
 \left( - \tau_{g}(t) + \tilde{\tau}_{g}(t) \, , \, \, \kappa_{n}(t) - \tilde{\kappa}_{n}(t)  \, ,
\, \, -\kappa_{g}(t) + \tilde{\kappa}_{g}(t)  \right)_{\tilde{\mathcal{F}}(t)}  .
\end{aligned}
\end{equation}
}
By comparing  \eqref{eqomegaTilde} with \eqref{eqCondRolling} we see that the conditions  \eqref{eqStandardRoll} are necessary and sufficient for $g_{t}$ to be a standard rolling.
\end{proof}

In passing we note -- for later use -- that \eqref{eqomegaTilde} and proposition \ref{prop:1} immediately give the coordinates of the pulled-back angular rotation vector  $\hat{\omega}_{t} = R_{t}^{\mathsf{T}}\omega_{t}$ with respect to the frame  ${\mathcal{F}}(t)$ for a  standard rolling:
\begin{equation}
\hat{\omega}_{t} = \Vert \gamma'(t) \Vert \cdot\left( - \tau_{g}(t) + \tilde{\tau}_{g}(t) \, , \, \, \kappa_{n}(t) - \tilde{\kappa}_{n}(t)  \, , \, \, 0  \right)_{{\mathcal{F}}(t)}  .
\end{equation}

The important special case in which $\tilde{S}$ is a plane is covered by the following corollary:

\begin{corollary} \label{corPlanarConseq}
If $\tilde{S}$ is a plane, then the motion $ g_{t}$ is a standard rolling if and only if
\begin{equation} \label{eqKappagn}
\begin{aligned}
\kappa_{g}(t) &= \tilde{\kappa}_{g}(t)\\
\kappa_{n}(t) & \neq 0  .
\end{aligned}
\end{equation}
The instantaneous angular rotation vector $\omega_{t}$ and its pull-back $\hat{\omega}_{t}$ are  correspondingly -- in $\tilde{\mathcal{F}}(t)$  and $\mathcal{F}(t)$ respectively:
\begin{equation} \label{eq:04}
\begin{aligned}
\omega_{t}&= \Vert \gamma^{\prime}(t) \Vert \cdot (-\tau_{g}(t), \kappa_{n}(t), 0)_{\tilde{\mathcal{F}}(t)} \\
&= \Vert \gamma^{\prime}(t) \Vert \cdot (-\tau_{g}(t), \kappa_{n}(t), 0)_{{\mathcal{F}}(t)}  ,
\end{aligned}
\end{equation}
where now  $\tilde{\mathcal{F}}(t) = \{ \tilde{e}(t), \tilde{e}_{3} \times \tilde{e}(t), \tilde{e}_{3}\}$ is the co-moving frame in the plane with constant normal vector field $\tilde{e}_{3}$ along $\tilde{\gamma}$.
\end{corollary}

\section{Developable Cartan surface ribbons} \label{secDevCartanRib}

In this section we show that the rolling discussed above serves as a tool for obtaining a flat developable approximation of the surface $S$ along $\gamma$. This is alternative to constructing developable approximations via envelopes of tangent planes along $\gamma$, see \cite[pp. 195-197]{DoCarmo}. In the recent work \cite{izumiya2015} osculating developable surfaces and their singularities have been studied, see also \cite{Izumiya2017}. It will follow from the condition \eqref{eqKappagn} that the approximating surface is free of singularities in a neighbourhood of $\gamma$, see theorem \ref{thmRibbonDevelop} below.

We first consider the notion of ruled surfaces, since developable surfaces constitute a special subcategory of those:

Let $w_{-}$ and $w_{+}$ denote two positive functions on the given $t$-interval $J$,  let $I = [-w_{-}(t), w_{+}(t)]$, and let $V$ denote the corresponding parameter domain in $\mathbb{R}^{2}$. A \emph{parametrized ruled surface (with boundary)} $r \colon V \to \mathbb{R}^{3}$ based on the center curve $\gamma$ is determined by a non-vanishing vector field $\beta$ along $\gamma$:
\begin{equation}
r(t,u) = \gamma(t) + u \cdot \beta(t), \quad t \in J, \quad u \in I.
\end{equation}
We will assume that $\beta$ is a unit vector field along $\gamma$ and that the surface $r$ is regular, i.e. its partial derivatives are linearly independent for all $u$ in the interval $[-w_{-}(t), w_{+}(t)]$,  $t \in J$. Regularity implies in particular that
\begin{equation} \label{eq:06}
\beta(t) \neq \pm e(t) \text{ for all } t \in J.
\end{equation}

Moreover, the surface $r(V)$ is flat (with Gaussian curvature zero at all points, i.e.~developable), precisely when the following condition is satisfied -- see \cite[p. 194]{DoCarmo}:
\begin{equation} \label{eq:07}
\beta' \cdot ( \beta \times e ) = 0.
\end{equation}

If $r(V)$ is eventually to be constructed so that it becomes a flat approximation of  $S$ along $\gamma$, we need to find a regular parametrization $r$ such that $r(V)$ is developable and has the same normal field $N$  as $S$ along $\gamma$. It means that we need  to determine the vector function $\beta$ so that it fulfills \eqref{eq:06}, \eqref{eq:07}, and
\begin{equation} \label{eq:08}
\beta \cdot N = 0.
\end{equation}

The desired vector function $\beta$ is precisely (modulo length and sign) the previously encountered pulled-back angular velocity vector $\hat\omega_{t}$ along $\gamma$ associated with the rolling of $S$ along $\tilde{\gamma}$ on a plane, see \cite{raffaelli2016}:

\begin{theorem} \label{thmRibbonDevelop}
Let $\gamma$ denote a smooth curve on a surface $S$ and let $\mathcal{F} = \{ e, h, N\}$ be the corresponding Darboux frame field along $\gamma$. Suppose that the normal curvature function $\kappa_{n}$ for $\gamma$ on $S$ never vanishes. Then there exists a unique developable surface which contains $\gamma$ and which has everywhere the same tangent plane as $S$ along $\gamma$. It is parametrized as follows:
\begin{equation}\label{eqCartanDevPar}
r(t,u) = \gamma(t) + u \cdot \frac{\hat{\omega}_{t}}{\Vert \hat{\omega}_{t} \Vert}  , \quad u \in [w_{-}(t), w_{+}(t)] , \quad t \in J.
\end{equation}
where $\hat{\omega}_{t}$ denotes the pulled-back angular velocity vector:
\begin{equation}
\begin{aligned}
\hat{\omega}_{t} &= \kappa_{n}(t)\cdot h(t) - \tau_{g}(t)\cdot e(t)  ,\\
\Vert \hat{\omega}_{t}\Vert &= \sqrt{\kappa_{n}^{2}(t) + \tau_{g}^{2}(t)}  .
\end{aligned}
\end{equation}
\end{theorem}

\begin{proof}
Write $\beta$ in terms of its coordinate functions $\beta \cdot e$ and   $\beta \cdot h$, substitute into equation \eqref{eq:07} and apply equation \eqref{eq:01} to express the derivatives of $e$ and $h$. Then
\begin{equation}
\frac{\beta \cdot e}{\beta \cdot h} = -\frac{\tau_{g}}{\kappa_{n}}  ,
\end{equation}
and the result follows upon normalization of the solution $\beta$. The ruling directions of the developable surface are thus given by the instantaneous angular velocity vector of the rolling.
\end{proof}

\begin{definition}
The developable surface, which is parametrized by \eqref{eqCartanDevPar} -- and which is therefore approximating the surface $S$ -- will be called the \emph{Cartan surface ribbon} along $\gamma$ on $S$.
\end{definition}

As is already in the name, the Cartan surface ribbon can be \emph{developed} isometrically into a planar ribbon:

\begin{definition}
 The \emph{associated Cartan planar ribbon} for $\gamma$ on $S$  -- which is defined along $\tilde{\gamma}$ in the plane -- is now determined via \eqref{eqPlanarCartan} in the proposition below, which also establishes the isometry between the two Cartan ribbons.
\end{definition}

\begin{proposition} \label{propIsometry}
An isometry from the Cartan surface ribbon onto the associated Cartan planar ribbon is realized along the development curve $\tilde{\gamma}$ in the following way, which is in precise accordance with the previously found rolling of $S$ along $\gamma$ on the plane with contact curve $\tilde{\gamma}$. We simply map the point $r(t,u)$ to the point
\begin{equation} \label{eqPlanarCartan}
\begin{aligned}
\tilde{r}(t, u) &= \tilde{\gamma}(t) + u \cdot \frac{{\omega}_{t}}{\Vert {\omega}_{t} \Vert} \\
&= \tilde{\gamma}(t) + u \cdot \frac{{\omega}_{t}}{\sqrt{\kappa_{n}^{2}(t) + \tau_{g}^{2}(t)}} .
\end{aligned}
\end{equation}
\end{proposition}
\begin{proof}
We let $\beta(t) = \omega_{t}/\Vert \omega_{t} \Vert$ and $\hat{\beta}(t) = \hat{\omega}_{t}/\Vert \hat{\omega}_{t} \Vert$. Since $\kappa_{g}(t) = \tilde{\kappa}_{g}(t)$ all the scalar products between two vectors chosen from $\{\gamma'(t), \hat{\beta}(t), \hat{\beta}'(t)\}$  are the same as the scalar products between the corresponding two vectors chosen from $\{\tilde{\gamma}'(t), {\beta}(t), {\beta}'(t)\}$. It follows that the two first fundamental forms for $r(t, u)$  and $\tilde{r}(t, u)$ respectively, have identical coordinate functions. The two ribbons $r$ and $\tilde{r}$ are therefore isometric.
\end{proof}


\begin{remark}
In all of the above constructions we have assumed that the center curves in question have nowhere vanishing normal curvature.
For a number of cases the normal curvature does vanish, such as on planar faces of polyhedra and through lines of inflections on generalized cylindrical faces. The method of approximation by  ribbons can be extended to these surfaces by cut and paste along the singular rulings under the condition that the geodesic torsion also vanishes together with the normal curvature.
For example, for surfaces containing planar domains, the ribbonization can be continued over any edge of the planar domain if the ruling of the ribbon agrees with the given edge. For polyhedral surfaces this is always possible. A ribbon with planar patches will also be denoted a Cartan ribbon, see the later section on Euler's polyhedral formula.

\end{remark}

\subsection{Curvature and parallel transport} \label{subsecCurvParal}

In view of our observations concerning the rolling of $S$ on the plane, it now makes sense to say that the Cartan surface ribbon can be \emph{rolled} isometrically onto the associated Cartan planar ribbon. This is induced in the way just described by the rolling of $S$ on the plane, which itself is represented by the pulled-back angular velocity vector field $\hat{\omega}$ along ${\gamma}$ in $S$ and by ${\omega}$ along $\tilde{\gamma}$ in the plane. Accordingly, once the center curve $\tilde{\gamma}$ in the plane has been constructed using $\tilde{\kappa}_{g}(t) = \kappa_{g}(t)$, then the approximating  Cartan \emph{surface} ribbon can be obtained via the inverse rolling of the Cartan \emph{planar} ribbon backwards into contact with the surface $S$ along $\gamma$. An early hint of this connection is presented in \cite[pp. 227-228]{SingerThorpe}.

The key object for the actual construction of the approximating Cartan surface ribbon along a given curve $\gamma$ on $S$ is thence the planar curve  $\tilde{\gamma}$, which may itself be constructed either by rolling, or -- simpler -- by integrating the curvature function $\kappa_{g}$  of $\gamma$, but in the plane, in the well known way, see \cite{DoCarmo}:

\begin{proposition}
Suppose $\tilde{\gamma}$ has (signed) curvature $\kappa_{g}$ and speed $\Vert \tilde{\gamma}'\Vert = v$. Then, modulo rotation and translation in the plane, we have:
\begin{equation} \label{eqCurvConstruct}
\tilde{\gamma}(t) = \int_{0}^{t}\,v(\hat{t})\cdot (\cos(\varphi(\hat{t}))\, , \,\, \sin(\varphi(\hat{t})) ) \, d\hat{t}
\end{equation}
where
\begin{equation}
\varphi(\hat{t}) = \int_{0}^{\hat{t}} v(\hat{u})\cdot \kappa_{g}(\hat{u}) \, d\hat{u}  .
\end{equation}
\end{proposition}

The curve $\tilde{\gamma}$ appears as a special -- and simple --  example of a \emph{Cartan development} as already alluded to via the reference to Nomizu's initial work, see \cite{nomizu1978}.
This is why the ensuing developable ribbons are called \emph{Cartan surface ribbons}. To be a bit more specific concerning our simple $2$-dimensional setting, we recall in particular the important geodesic curvature equivalence used above:

We let the tangent space $T_{\gamma(0)}S$ at $\gamma(0)$ represent the plane $\tilde{S}$ into which we want to construct the Cartan development curve corresponding to the given curve $\gamma$ in $S$. For each $t$ we consider the parallel transport of the tangent vector $\gamma'(t)$ along  $\gamma$ from the point $\gamma(t)$ to the point $\gamma(0)$, see
\cite[p. 131]{Kobayashi1963}:
\begin{equation}
X(t) = \Pi_{\gamma}^{\gamma(t)\, ,\, \gamma(0)}\left( \gamma'(t)\right)  .
\end{equation}
The Cartan development $\tilde{\gamma}$ of $\gamma$ in $T_{\gamma(0)}S$ is then:
\begin{equation}
\tilde{\gamma}(t) = \int_{0}^{t} \, X(u)\,\, du .
\end{equation}
From this construction it follows in particular that
\begin{proposition}
Any tangent vector $\tilde{\gamma}'(t_{1}) = X(t_{1})$ is itself parallelly transported (in the usual Euclidean sense) along $\tilde{\gamma}$ in the tangent space $T_{\gamma(0)}S$ (which may be canonically identified with $T_{\tilde{\gamma}(0)}\tilde{S}$) from $(0,0)$ to $\tilde{\gamma}(t_{1})$ and the (geodesic) curvature function of the planar curve $\tilde{\gamma}$ is equal to the geodesic curvature function of the original curve $\gamma$ in $S$:
\begin{equation}
\tilde{\kappa}_{g}(t) = \kappa_{g}(t) \quad \textrm{for all $t$.}
\end{equation}
\end{proposition}

\begin{proof}
Suppose $Y$ is any parallel vector field along the curve $\gamma$ on the surface $S$, then the angle $\theta(t) = \angle(Y(t), \gamma'(t))$ gives the geodesic curvature of $\gamma$ via $\theta'(t) = \kappa_{g}(t)$. Since the same holds true by construction along the development curve $\tilde{\gamma}$ in the tangent plane, we get $\tilde{\theta}(t) = \theta(t)$ so that $\tilde{\kappa}_{g} = \kappa_{g}$.
\end{proof}

\subsection{A measure of local goodness of Cartan ribbon approximations} \label{subsecGoodness}

A measure of the goodness of a single ribbon approximation along a given center curve $\gamma$ can be obtained from the following construction.
Close to $\gamma$ the surface $S$
can be parametrized as a graph surface 'over' the Cartan ribbon in the direction of \emph{the normal field $N$ of the ribbon} as follows:
\begin{equation}
S_{\varepsilon}\, : \,\,\,  \sigma(t,u) =  \gamma(t) + u \cdot \left( \frac{\kappa_{n}(t) h(t) - \tau_{g}(t) e(t)}{\sqrt{\kappa_{n}^{2}(t) +\tau_{g}^{2}(t) }}\right) + f(t,u)\cdot N(t)  \,  ,  \, t \in J \, ,  \, u \in [-\varepsilon, \varepsilon]  ,
\end{equation}
where $f$ denotes the corresponding 'height' function and $\varepsilon$ is everywhere smaller than each of the width functions $w_{-}$ and $w_{+}$ for all $t \in J$ along $\gamma$. (Both width functions have positive minima since they are positive and $J$ is closed.) The function $f$ clearly has $f(t,0) = f'(t,0) = 0$ for all $t \in J$, so that
\begin{equation}
f(t,u) = \frac{1}{2}f''(t,0)\cdot u^{2} + O(u^{3}) \,\, \textrm{for each} \,\, t \in J \, \, \textrm{and for all} \, \, u \in [-\varepsilon, \varepsilon]  .
\end{equation}

The domain in space that is enclosed 'between' the surface $S_{\varepsilon}$ and the Cartan Ribbon is thence parametrized as follows:
\begin{equation}
\begin{aligned}
\mathcal{D}_{\varepsilon}\, : \, \, R(t,u,w) &= \gamma(t) + u \cdot \left(\frac{\kappa_{n}(t) h(t) - \tau_{g}(t) e(t)}{\sqrt{\kappa_{n}^{2}(t) +\tau_{g}^{2}(t) }}\right) + w \cdot f(t,u) \cdot N(t)  ,\\
\textrm{where} \, \, t &\in J \, \, , \, \,  u \in [-\varepsilon, \varepsilon] \, \, , \, \, w \in [0, 1]  .
\end{aligned}
\end{equation}

\begin{definition}
We consider the volume of the domain $\mathcal{D}_{\varepsilon}$ as a natural local measure of goodness $\mathcal{M}(\gamma, \varepsilon)$ of our approximation of the surface $S$, i.e. of the approximation by the single Cartan ribbon to the tubular neighborhood $\mathcal{S}_{\varepsilon}$ of width $2\varepsilon$ along the center curve $\gamma$:

\begin{equation}
\mathcal{M}(\gamma, \varepsilon) = \operatorname{Vol}(\mathcal{D}_{\varepsilon}) = \int_{J} \, \int_{-\varepsilon}^{\varepsilon} \, \int_{0}^{1} \,
\left|(R'_{t} \times R'_{u}) \cdot R'_{w} \right| \, dt \, du \, dw .
\end{equation}
\end{definition}

We then have the following evaluation of $\mathcal{D}_{\varepsilon}$.

\begin{theorem} \label{thmGoodness}
The goodness $\mathcal{M}(\gamma, \varepsilon)$ of the single ribbon approximation along a unit speed center curve $\gamma$ can be expressed  in terms of the curvature functions $H(t)$, $K(t)$, $\kappa_{n}(t)$ and $\tau_{g}(t)$ along $\gamma$ as follows:
\begin{equation} \label{eqGoodness}
\mathcal{M}(\gamma, \varepsilon) =   \frac{1}{3}\varepsilon^{3}\cdot \int_{J}\,F(H(t), K(t), \kappa_{n}(t), \tau_{g}(t)) \, dt + O(\varepsilon^{4})  ,
\end{equation}
where
\begin{equation}
F(H, K, \kappa_{n}, \tau_{g})
= \frac{ \kappa_{n}^{2}}{\left(\kappa_{n}^{2}
+\tau_{g}^{2}\right)^{3/2} } \cdot \left| \left( \tau_{g}^{2} - \kappa_{n}^{2} +2H\kappa_{n} - 2\tau_{g} \sqrt{2H\kappa_{n} - K - \kappa^{2}_{n}}\right)  \right|  .
\end{equation}
\end{theorem}

\begin{proof}
Using the parametrization of $\mathcal{D}_{\varepsilon}$ and the derivatives of the Darboux frame in \eqref{eq:01} we find that the volume element  $\left|(R'_{t} \times R'_{u}) \cdot R'_{w}\right|$ has the following leading term:
\begin{equation} \label{eqVolElem}
\left|(R'_{t} \times R'_{u}) \cdot R'_{w}\right| =  \left| \frac{1}{2} f''_{uu}(t,0) \cdot \frac{ u^{2} \cdot \kappa_{n}(t)}{\sqrt{\kappa_{n}^{2}(t) +\tau_{g}^{2}(t) }} \right| + O(u^{3}) .
\end{equation}

The second derivative $f''_{uu}(t,0)$ is precisely the normal curvature of the surface $S$ in the \emph{direction of the ruling line} of the Cartan ribbon at $\gamma(t)$. It can thence be expressed by the curvature function values $H(t)$, $K(t)$, $\kappa_{n}(t)$ and $\tau_{g}(t)$ at $\gamma(t)$ along $\gamma$:

\begin{equation} \label{eqDoublederiv}
f''_{uu}(t,0) = \frac{\kappa_{n}}{\kappa^{2}_{n} + \tau^{2}_{g}} \left( \tau_{g}^{2} - \kappa_{n}^{2} +2H\kappa_{n} - 2\tau_{g} \sqrt{2H\kappa_{n} - K - \kappa^{2}_{n}}\right)  .
\end{equation}

Insertion into \eqref{eqVolElem} then gives:

\begin{equation*}
\begin{aligned}
&\mathcal{M}(\gamma, \varepsilon) = \int_{J}\,\int_{-\varepsilon}^{\varepsilon} \, \int_{0}^{1} \left|(R'_{t} \times R'_{u}) \cdot R'_{w}\right| \, dt \, du \, dw \\
&= \int_{J}\,\int_{-\varepsilon}^{\varepsilon}\, \left| \frac{1}{2} \cdot \frac{u^{2} \cdot \kappa_{n}^{2}}{\left(\kappa_{n}^{2}
+\tau_{g}^{2}\right)^{3/2} }\left( \tau_{g}^{2} - \kappa_{n}^{2} +2H\kappa_{n} - 2\tau_{g} \sqrt{2H\kappa_{n} - K - \kappa^{2}_{n}}\right)  \right| + O(u^{3})\, dt \, du \\
&= \frac{1}{3}\varepsilon^{3} \cdot \int_{J}\,\left|\frac{ \kappa_{n}^{2}}{\left(\kappa_{n}^{2}
+\tau_{g}^{2}\right)^{3/2} }\left( \tau_{g}^{2} - \kappa_{n}^{2} +2H\kappa_{n} - 2\tau_{g} \sqrt{2H\kappa_{n} - K - \kappa^{2}_{n}}\right)  \right|  \, dt + O(\varepsilon^{4}) \\
&= \frac{1}{3}\varepsilon^{3}\cdot \int_{J}\,F(H(t), K(t), \kappa_{n}(t), \tau_{g}(t)) \, dt + O(\varepsilon^{4})  .
\end{aligned}
\end{equation*}
\end{proof}

\begin{corollary} \label{corCurvLines}
Suppose that the center curve $\gamma$ is a line of curvature on the surface $S$ -- as is the case for all the chosen center curves on the ellipsoid considered in section \ref{secExampEllipsoid} below. Then the geodesic torsion of $\gamma$ vanishes identically and the corresponding local measure of goodness of the Cartan ribbon along $\gamma$ reduces to:
\begin{equation}
\mathcal{M}(\gamma, \varepsilon) =   \frac{1}{3}\varepsilon^{3}\cdot \int_{J}\,\left|\kappa_{n}(h(t))\right|\, dt + O(\varepsilon^{4})  ,
\end{equation}
where $\kappa_{n}(h(t))$ denotes the normal curvature of $S$  at $\gamma(t)$ in the direction of $h(t)$, which is orthogonal to $\gamma'(t)$.
\end{corollary}
\begin{proof}
This follows directly from equation \eqref{eqVolElem} and the fact that in this case we have $f''_{uu}(t,0) = \kappa_{n}(h(t))$.
\end{proof}

Another consequence of theorem \ref{thmGoodness} is the following result, which is not surprising, since we are approximating the surface $S$ with flat Cartan ribbons:

\begin{corollary}
Suppose that the Gaussian curvature $K$ of $S$ vanishes identically along $\gamma$. Then
\begin{equation}
\mathcal{M}(\gamma, \varepsilon) =  O(\varepsilon^{4})  .
\end{equation}
\end{corollary}
\begin{proof}
This follows readily by inserting the following ingredients into the formula \eqref{eqGoodness}:
\begin{align*}
K(t) &= 0 , \\
H(t) &= \kappa_{1}(t) , \\
\kappa_{2}(t) &= 0 , \\
\tau_{g}(t) &= \kappa_{1}(t)\cos(\theta(t))\sin(\theta(t)) ,\\
\kappa_{n}(t) &= \kappa_{1}(t)\cos^{2}(\theta(t)) ,
\end{align*}
where $\theta(t)$ denotes the angle between $\gamma'(t)$ and the principal direction of curvature for $S$ at $\gamma(t)$ corresponding to the principal curvature $\kappa_{1}(t)$.
\end{proof}

\begin{remark}
Although theorem \ref{thmGoodness} is but an initial step towards a global measure of goodness for the total number of individual Cartan ribbons (that are in use for the overall approximation of a given full surface), it may still be possible and reasonable to apply the formula \eqref{eqGoodness} -- or a proper refinement of it -- for each ribbon and then simply sum the values of goodness over the number of ribbons. Naturally, the $u$-domain of integration should then not just be $[-\varepsilon, \varepsilon]$ but rather the full width-interval $[-w_{-}(t), w_{+}(t)]$ along the respective ribbons.
Moreover, good single ribbon approximations (and their higher dimensional analogues) represent an interesting alternative basis and tool for principal geodesic analysis, and for polynomial regression in general, on surfaces and in Riemannian manifolds, see \cite{fletcher2004a} and \cite{hinkle2012a}. In particular, in that setting the notion of Riemannian polynomials have also been studied via rolling maps, see \cite{jupp1987a} and \cite{leite2015a} -- much in the same vein as we have employed the concept of rolling in the present work.
\end{remark}

\subsection{The local cut-off procedure for neighboring ribbons} \label{subsecIntersect}

We consider two neighboring center curves $\gamma^{1}$ and $\gamma^{2}$ for two neighboring Cartan ribbons and prove the existence of their intersection curve, that eventually constitute the wedge (or cut-off) curve in space 'between' the two center curves, see the examples in sections \ref{secExampTorus} and \ref{secExampEllipsoid}. The wedge thereby defines the actual width functions $w^{2}_{-}$ and $w^{1}_{+}$, that are used for the final ribbonization of the surface $S$. In this setting $w^{1}_{+}$ is to be thought of as the cut off function for $\gamma^{1}$ in the direction towards $\gamma^{2}$, and $w^{2}_{1}$ is the cut off function for $\gamma^{2}$ in the (opposite) direction from $\gamma^{2}$  towards $\gamma^{1}$.

\begin{proposition}
The wedges are well-defined for each pair of neighboring Cartan ribbons, i.e. the cut-off functions exist, provided the corresponding center curves are pairwise sufficiently close to each other.
\end{proposition}

\begin{proof}
We sketch the proof as follows.
Suppose that $r_{1}$ is the ruling line at some point $p$ on $\gamma^{1}$. We must show that (for close-by neighboring center curves) there is a corresponding ruling line $r_{2}$ at some point of $\gamma^{2}$ so that the two rulings intersect in a (cut-off) point, i.e. so that $w^{1}_{+}$ and $w^{2}_{-}$ exist. Obviously, this does not necessarily work for center curves that are far apart from each other, so we need that the center curves are sufficiently close.

We may assume that the two center curves are neighboring coordinate curves in a special local parametrization of
 a tubular neighborhood around $\gamma^{1}$. Specifically, without lack of generality, we parametrize the neighborhood by a smooth vector function $\rho$ with parameters $t$ and $v$ such that the following properties are satisfied: $\rho(t,0) = \gamma^{1}(t)$; $\rho(t,\varepsilon) = \gamma^{2}(t)$; every $t$-coordinate curve has nonvanishing normal curvature:
  $\kappa_{n}(\rho'_{t}(t,v)) \neq 0$; and  $\rho'_{v}(t_{0}, v)$ is in the direction of the ruling line of the Cartan ribbon along the curve $\rho(t, v)$ at the point $\rho(t_{0}, v)$ for all $v \in [0, \varepsilon]$.

  This latter condition means that the curve $q_{t_{0}}(v) = \rho(t_{0}, v)$, $v \in [0, \varepsilon]$ has tangent lines that are ruling lines of the respective Cartan ribbons along the center curves $\rho(t, v)$ for each $v$ in the said interval.

If the curve  $q_{t_{0}}$ has nonzero curvature at $v = 0$ (and possibly also nonzero torsion there), then an intersection argument in the ambient space shows that there exists a ruling line of the Cartan ribbon at some point $\rho(t_{2}, \varepsilon)$ along the center curve $\rho(t, \varepsilon)$  close to  $\rho(t_{0}, \varepsilon)$, i.e. for $t_{2}$ close to $t_{0}$, which intersects the ruling line $r_{1}$ based at $p = \rho(t_{0}, 0)$ --  provided $\varepsilon$ is sufficiently small. If the torsion of the curve $q_{t_{0}}$ vanishes in the interval $v \in [0, \varepsilon]$ so that it is planar in that interval, then $t_{2} = t_{0}$ and the intersection takes place in that plane.

The same argument holds if $q_{t_{0}}$ has zero curvature at $v = 0$ but, say, has positive curvature for $ v \in ]\delta, \varepsilon]$. Moreover, if $q_{t_{0}}$ has zero  curvature in an interval, $v \in [0, \varepsilon[$, then $q_{t_{0}}$ is a straight line in that interval and every point on the ruler from $p$ is also a point on a ruler for the ribbon with center curve $\eta(t, v_{0})$ for any $v_{0}$ in that interval, and the corresponding cut-off value for $w^{1}_{+}$ can be chosen to be any value in $]0, \varepsilon[$.
\end{proof}

\section{Gauss--Bonnet inspection} \label{secGBinspect}

We consider a finite (piecewise smooth) ribbonization $\mathcal{R} = \cup_{i}^{R} \mathcal{R}_{i}$, $R= \# \mathcal{R}$, of $S$ all of whose Cartan surface ribbons $\mathcal{R}_{i}$, $i = 1, \ldots, R$,  are closed in the sense that they are based on closed smooth center curves on $S$ as in figures \ref{figTorusRib}
and \ref{figEllipsoidRib1}
below. Let $\mathcal{W} = \cup_{i} \mathcal{W}_{i}$ denote the system of (piecewise smooth) wedge curves stemming from the ribbonization $\mathcal{R}$ and let $\widehat{\mathcal{W}}$ denote the corresponding planar wedge curve system of the Cartan planar ribbons $\widehat{\mathcal{R}}$. The end (cut-)curves of the planar ribbons -- that are typically needed in order to obtain the planar representation of the ribbons -- are not considered part of $\widehat{\mathcal{W}}$.

We now apply the Gauss--Bonnet theorem to surfaces which are ribbonized by such circular ribbons.

The system of wedge curves $\mathcal{W}$ consists of curves with possible branch-points, where three or more ribbons come together, and with  possible end-points, where one ribbon is locally bent around the wedge (and is thus in contact with itself), as in the top and bottom ribbon on the ellipsoid in
Fig.~\ref{figEllipsoidRib1} below.

We may assume without lack of generality that the branch points and end points are all isolated and regular in the sense that the wedge curves in a neighbourhood of such points can be mapped diffeomorphically to a corresponding star configuration in $\mathbb{R}^{3}$ with a number of straight line segments issuing from a common vertex. The branch-points and end-points are called \emph{vertices} of the ribbonization $\mathcal{R}$. The vertex set is denoted by $\mathcal{P}$ and the number of vertices by $P = \# \mathcal{P}$. The number of  segments issuing from a given vertex $p_{k}$ in the vertex set $\mathcal{P}$ is called the \emph{degree}, $d_{k} = d(p_{k})$ of the vertex. If a ribbon has an isolated cone point then this is also a vertex, and -- in accordance with the above definition -- we count its degree as $0$.

\begin{theorem} \label{theoremE1}
The Euler characteristic, $\chi (\mathcal{R})$, of a ribbonization $\mathcal{R}$ is
 \begin{equation} \label{eqE1}
\chi (\mathcal{R})= \frac{1}{2} \sum_{k=1}^{P} \left(2 - d_{k} \right)  .
\end{equation}
\end{theorem}

\begin{proof}
The total curvature contributions for the Gauss--Bonnet theorem can be divided into three parts:\\

a) \emph{Surface contributions:} the surface integral of the Gauss curvature $K$,
\begin{equation}
 C_\mathcal{R\backslash W\cup P} = \int_{\mathcal{R\backslash W \cup P}} K d\mu = 0  .
 \end{equation}

b) \emph{Wedge contributions:}  The integral of the geodesic curvature along the edges of the Cartan ribbons excluding the vertex points,

\begin{equation}
C_\mathcal{W\backslash P} = \sum_{q=1}^{R} \,
\int_{{\mathcal{W}}_{q}\backslash \mathcal{P}_{q}} \kappa^{{\mathcal{W}}_{q}}(s) \, ds  .
\end{equation}

c) \emph{Vertex contributions:} sum of the angular deficit (angular defect) at the vertices, i.e. $2\pi$ minus the sum of the inner angles $\beta (j,k)$ at the vertices. The inner angles are replaced by the corresponding outer angles $\alpha(j,k)$ as $\alpha=\pi-\beta$ where $\alpha \in [-\pi, \pi]$ and $\beta \in [0,2\pi]$,

\begin{equation}
\begin{split}
C_\mathcal{P} & = \sum_{j=1}^{P} \, \left( 2\pi -\sum_{k=1}^{d_{k}} \beta(j,k) \right) \\
&=\sum_{k=1}^{P} \, \left( 2\pi -\sum_{j=1}^{d_{k}} (\pi - \alpha(j,k)) \right) \\
&=\sum_{k=1}^{P} \, \left( 2\pi -\pi d_{k}\right) + \sum_{k=1}^{P} \sum_{j=1}^{d_{k}} \alpha(j,k)  .
\end{split}
\end{equation}

\noindent \emph{Summarising}: Adding these contributions together we find:

\begin{equation}
2\pi \cdot \chi (\mathcal{R})=
\sum_{q=1}^{R} \,
\int_{{\mathcal{W}}_{q}\backslash \mathcal{P}_{q}} \kappa^{{\mathcal{W}}_{q}}(s) \, ds +
 \sum_{k=1}^{P} \sum_{j=1}^{d_{k}} \alpha(j,k)  + \sum_{k=1}^{P} (2\pi -\pi d_{k})  .
\end{equation}

By a permutation of the outer angles in the second term one can group them according to the ribbon wedge curves they appear on. This is possible because each of the kinks on the ribbons is encountered precisely once in the summation. Further, as the ribbons are closed it follows that their wedge integral and the corresponding sum of outer angles together cancels to zero. Hence one is left with the equality:
\begin{equation} \label{eqChi}
\chi=\frac{1}{2} \sum_{k=1}^{P}(2-d_{k})  .
\end{equation}

\end{proof}

\begin{remark}
As mentioned, the set of vertices, $\mathcal{P}$, is a feature of the three-dimensional mesh of wedge curves. Wedge curves from two, most commonly distinct, ribbons follow each other until a vertex point, where, e.g.~three ribbons come together. We summarise the different vertex characters with Table \ref{quick1}.

\begin{table}[h!]
\centering
\begin{tabular}{|l|l|l|}
  \hline
  $d_{k}$ & Vertex character & Classification \\
  \hline \hline
  0 & Fully circumscribed by one ribbon & Cone point\\
  1 & Half circumscribed by a ribbon & Wedge end point\\
  2 & Two ribbons meet at the point & Zero contributing vertex\\
   $>2$ & $n>2$ ribbons meet & Conventional vertex\\
  \hline
\end{tabular}
\vspace{.3cm}
\caption{A characterisation of vertices.}
\label{quick1}
\vspace{-0.50cm}
\end{table}

\end{remark}

The ribbon formula in Theorem \ref{theoremE1} is valid for orientable as well as non-orientable surfaces. To see this we only need to show that the formula does not change whether the ribbons are regular closed ribbons or M\"{o}bius strip-ribbons. This follows as a consequence of lemma \ref{lemMobius} below.

\begin{lemma} \label{lemMobius}
A conventional cylindrical closed ribbon (without vertices), and a M\"{o}bius strip-ribbon both contribute zero to the total curvature integral.
\end{lemma}

\begin{proof}
It follows simply by cutting the ribbons along a ruler. In this case, the ribbons can be fully flattened and has a total curvature contribution of $2\pi$ which is equal to the sum of the four artificial angles introduced by the cutting along the ruler. The difference between a ribbon that is orientable and one that is not consists of a simple permutation of the four inner angles of the cut.
\end{proof}

For the explicit extension of theorem \ref{theoremE1} to  Cartan ribbonizations that include ribbons with open center curves, it is sufficient to count the number $N_{O}$ of such ribbons, and equation \ref{eqE1} becomes:
\begin{equation}
\chi=\frac{1}{2} \sum_{k=1}^{P}(2-d_{k}) + N_{O}  .
\end{equation}

\begin{remark}
A necessary and sufficient criterion for the correct representation of the topology of the surface $S$ by a given ribbonization is the following: For each ribbon there exists a  homeomorphism which maps the ribbon to a domain on the surface such that
\begin{enumerate}
\item The contact structure (edges and vertices) between the individual ribbons is preserved
\item The full surface $S$ is covered precisely once by the images of the ribbons.
\end{enumerate}
For ribbonizations with sufficiently narrow ribbons, i.e.  with small cut-off functions $w_{-}$ and $w_{+}$, such homeomorphisms can for example be obtained via normal projection (along the orthogonal lines to $S$) of the ribbons into the surface.
\end{remark}

\subsection{From ribbon inspection to the Euler polyhedral formula} \label{subsecRibbon2Euler}

We consider a polyhedron $Q$ and apply the conventional notation, i.e.~$F$, $E$ and $V$ denote the number of faces, of edges, and of vertices, respectively, of the \emph{polyhedron}. To apply the ribbon formula \eqref{eqE1} we need to cover the polyhedron with closed ribbons. One can cover each one of the $F$ faces by a closed ribbon with a (flat) vertex covering the intrinsic part of the face polygon. With this choice there are then $F$ new such virtual vertices, all with degree zero. We therefore have the total number of ribbon vertices

\begin{equation}
P = V+F
\end{equation}
and
\begin{equation}
\sum_{k=1}^{P} d_{k} = 2\cdot E  .
\end{equation}

Hence we recover the well known polyhedron formula from the ribbon formula:

\begin{equation}
\chi(Q) = \frac{1}{2} \sum_{k=1}^{V}(2 - d_{k} ) =\frac{2(V+F)-2 E}{2} =V-E+F .
\end{equation}



\section{An unknot-based Cartan ribbonized torus} \label{secExampTorus}


This example is concerned with the ribbonization of the torus
\begin{equation}
\mathcal{T}^{2}:  \, \, \sigma(u,v)= \big((2+ \cos(u))\cdot \cos(v)\, ,\, (2+ \cos(u))\cdot \sin(v)\, , \, \sin(u) \big) \,  , \, (u, v) \in \mathbb{R}^{2}
\end{equation}
using the following two closed curves as center curves (see Fig.~\ref{figTorusKnots}):
\begin{equation} \label{eqTorusCenterCurves}
\begin{aligned}
\gamma_{1}(t) &= \sigma\left(3\cdot t\, , \, \, t \right)  , \quad t \in [-\pi, \pi]  \\
\gamma_{2}(t) &= \sigma \left(3 \cdot t \, , \, \,  t + \frac{\pi}{3} \right)  , \quad  t \in [-\pi, \pi] .
\end{aligned}
\end{equation}

\begin{figure}[h!]
\begin{center}
\includegraphics[width=60mm]{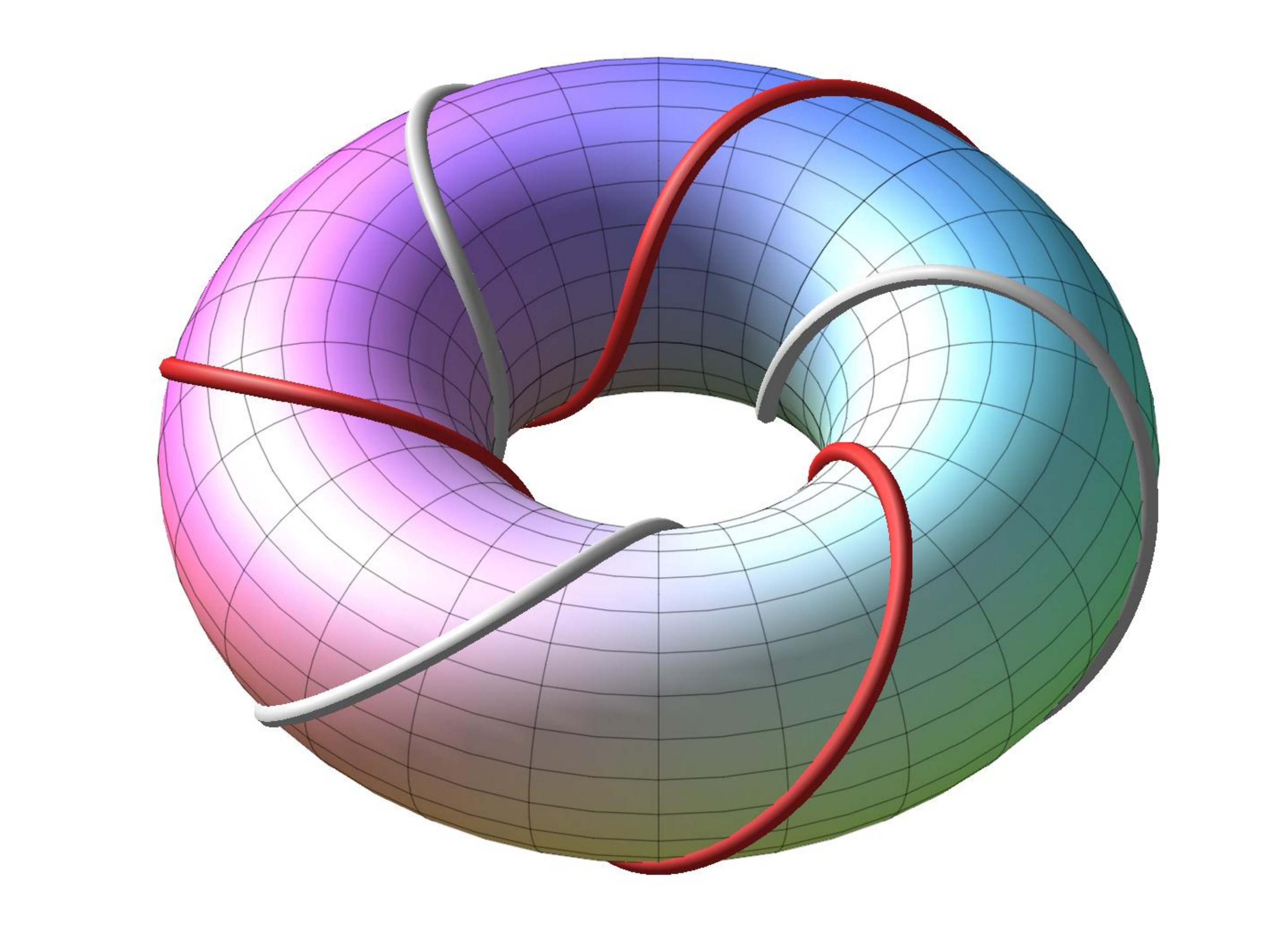}\qquad \includegraphics[width=53mm]{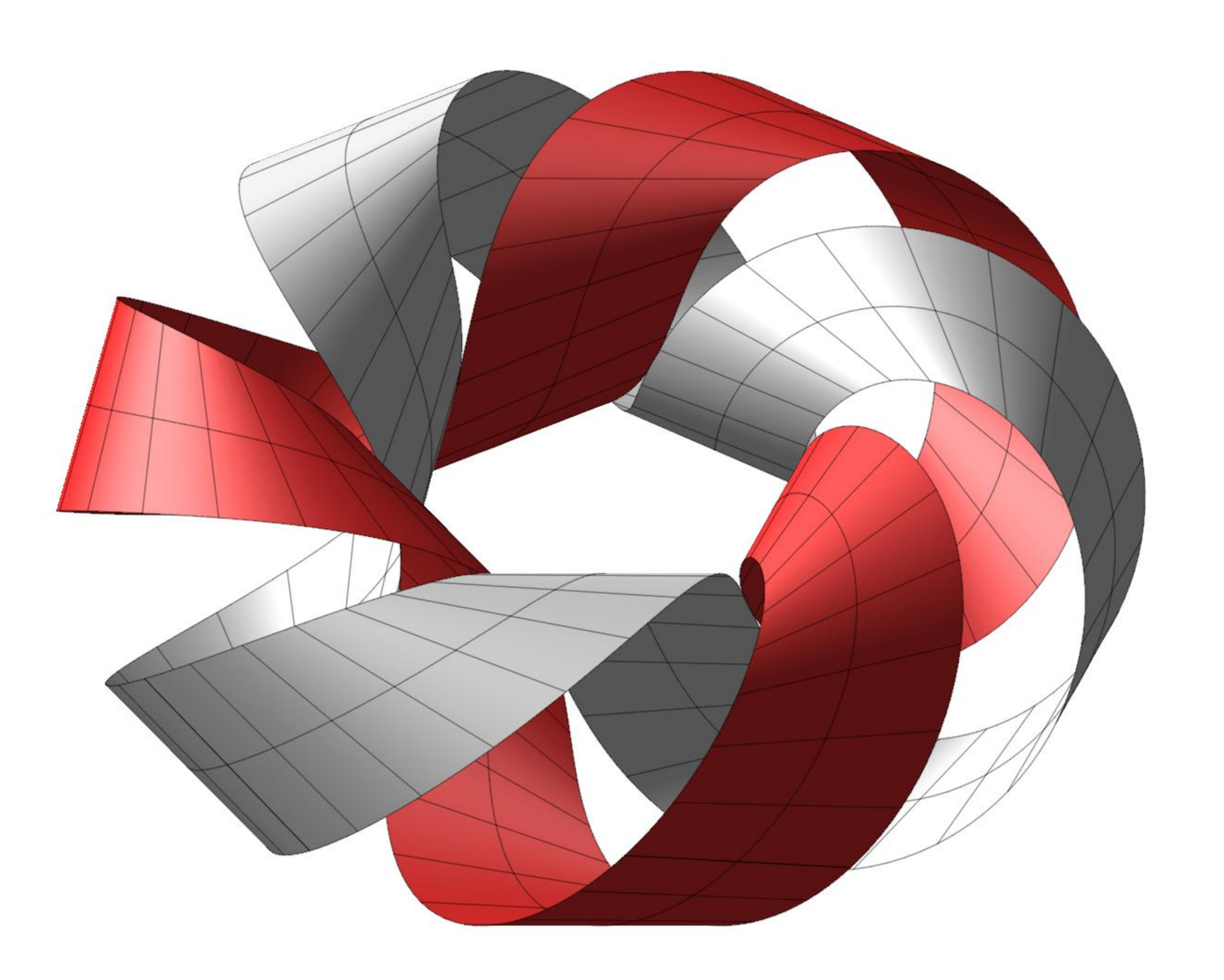}
\caption{Two $(3,1)$-unknots on a torus that are used as center curves for the beginning of a Cartan ribbonization of the torus.
See Fig.~\ref{figTorusRib} with the corresponding  ribbons, extended and cut-off.} \label{figTorusKnots}
\end{center}
\end{figure}

The corresponding two Cartan surface ribbons are then constructed (with constant and equal width functions) along the two curves, using the parametrization recipe in \eqref{eqCartanDevPar}. They are displayed on the right in Fig.~\ref{figTorusKnots}. The ribbons are then widened in $\mathbb{R}^{3}$ in the direction of $\pm \omega$ until intersection with their respective neighbour ribbons.

\begin{figure}[h!]
\begin{center}
\includegraphics[width=40mm]{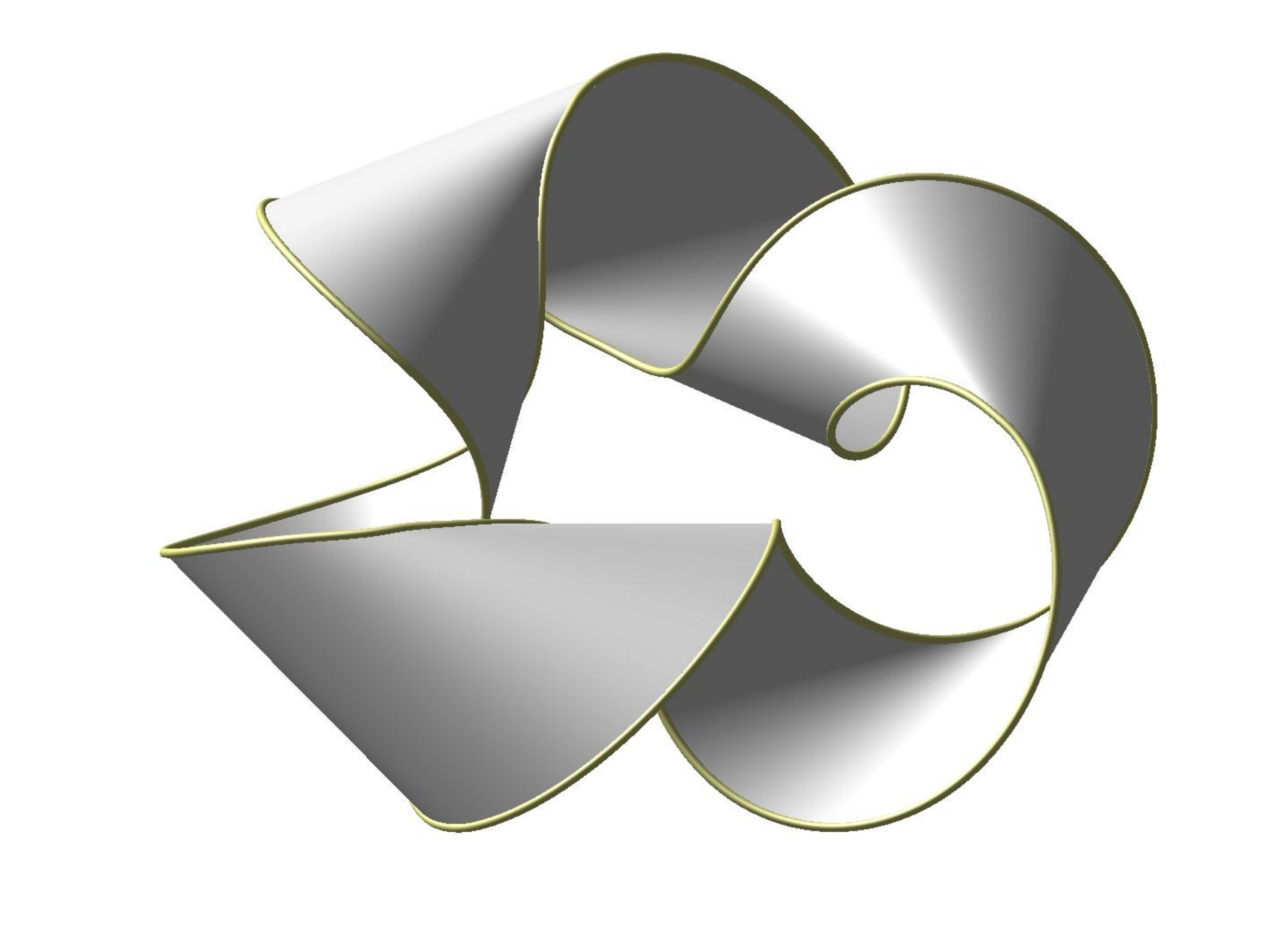}\quad \includegraphics[width=40mm]{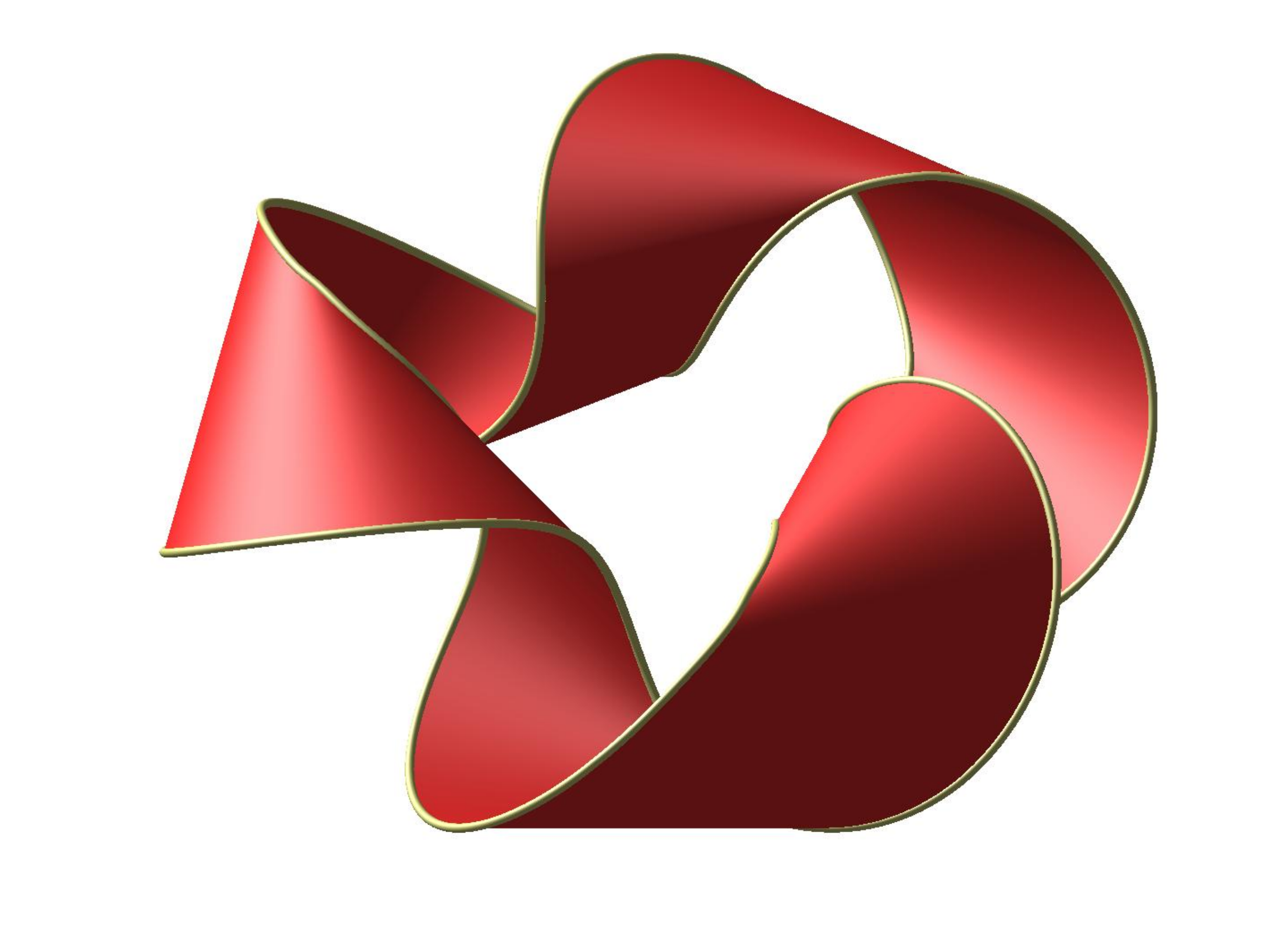}\quad \includegraphics[width=40mm]{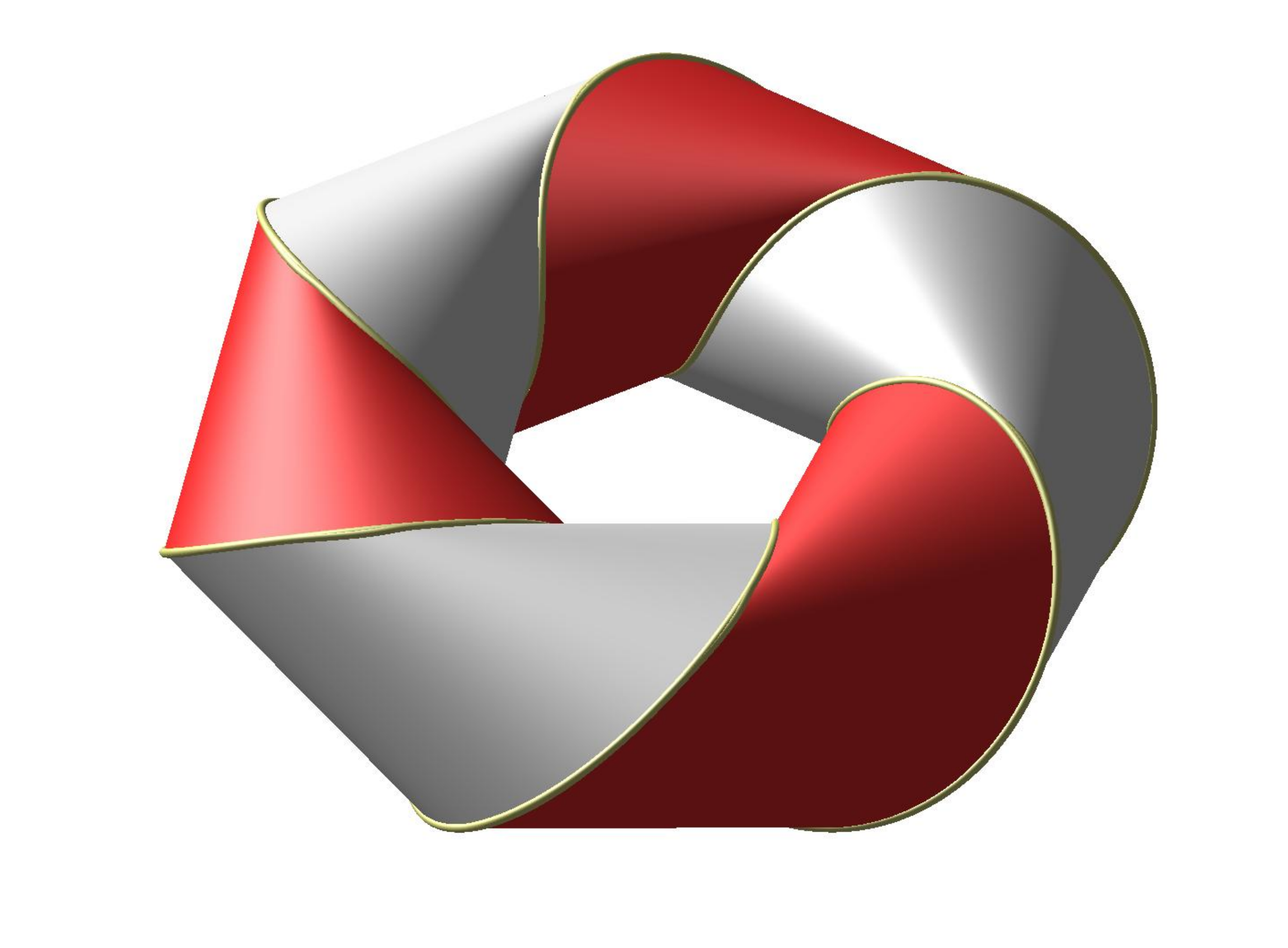}
\caption{Ribbonization of the torus along two $(3,1)$-unknots with the correct cut-off width functions.} \label{figTorusRib}
\end{center}
\end{figure}

In the present example the planar ribbons are constructed via the planar center curves $\tilde{\gamma}$ from \eqref{eqCurvConstruct} using the geodesic curvature function from the curves \eqref{eqTorusCenterCurves} on the torus, see figure \ref{figTorusRib2}.

\begin{figure}[h!]
\begin{center}
\includegraphics[width=110mm]{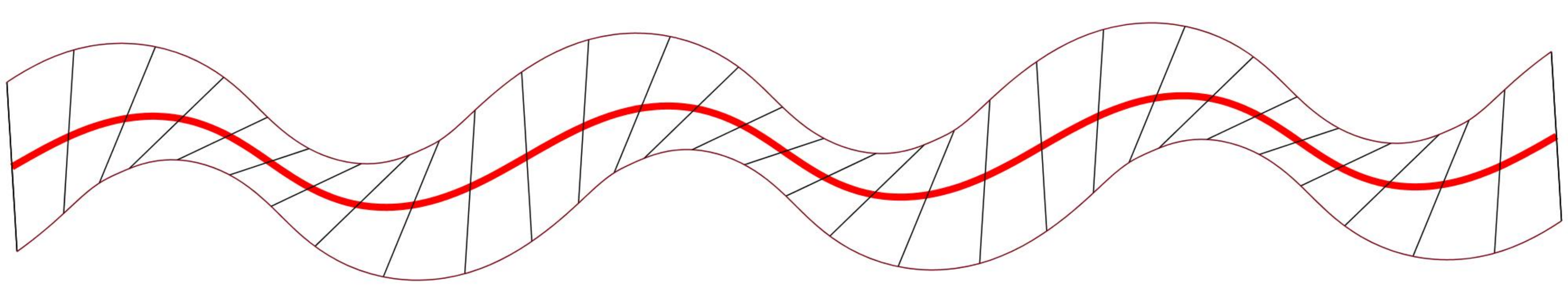}
\caption{The geometry of one of the two identical planar ribbons used for the covering of the torus in Fig.~\ref{figTorusRib}.} \label{figTorusRib2}
\end{center}
\end{figure}

The intersection width functions are obtained numerically by solving the intersection equation for each value of $t$ along the center curves, see Fig.~\ref{figTorusRib}.
Once the cut-off widths $w_{\pm}$ of the Cartan surface ribbons have been determined, the corresponding Cartan planar ribbons (with the same width-functions $w_{\pm}(t)$) are finally constructed from the planar center curve with the same geodesic curvature as the original center curve on the surface. In this particular case both Cartan planar ribbons are identical -- one of them is displayed in
Fig.~\ref{figTorusRib2}.

\subsection{Inspection of the ribbonized torus}
The number of vertices of the above ribbonization is $0$ and hence according to equation \eqref{eqChi} we get immediately  Euler characteristic $\chi = 0$ for the torus.

\section{Curvature line based ribbonizations of an ellipsoid} \label{secExampEllipsoid}

A curvature line parametrization of the ellipsoid with half axes $\sqrt{a}\,\, > \, \, \sqrt{b}\,\, >\,\, \sqrt{c} \, \, > 0$ is obtained as follows, see \cite{Sotomayor2008a} and \cite[Example 7.4]{hananoi2017}:

\begin{equation}
\label{eqEllipParam}
\sigma(u,v) = \Big( \pm \sqrt{\frac{a(a-u)(a-v)}{(a-b)(a-c)}}, \\
\pm \sqrt{\frac{b(b-u)(b-v)}{(b-a)(b-c)}},
\pm \sqrt{\frac{c(c-u)(c-v)}{(c-a)(c-b)}} \,\, \Big),
\end{equation}
where $u \in \, (b, a)$ and $v \in \, (c,b)$. This particular parametrization of the ellipsoid is shown in the leftmost display in Fig.~\ref{figEllipsoidRib1}. As shown on the display the coordinate (curvature) lines of this parametrization extend smoothly from one octant to a neighbouring octant except at the $4$ umbilical points on the ellipsoid corresponding to parameter values $u \to b$ and $v \to b$.

\begin{figure}[h!]
\begin{center}
\includegraphics[width=30mm]{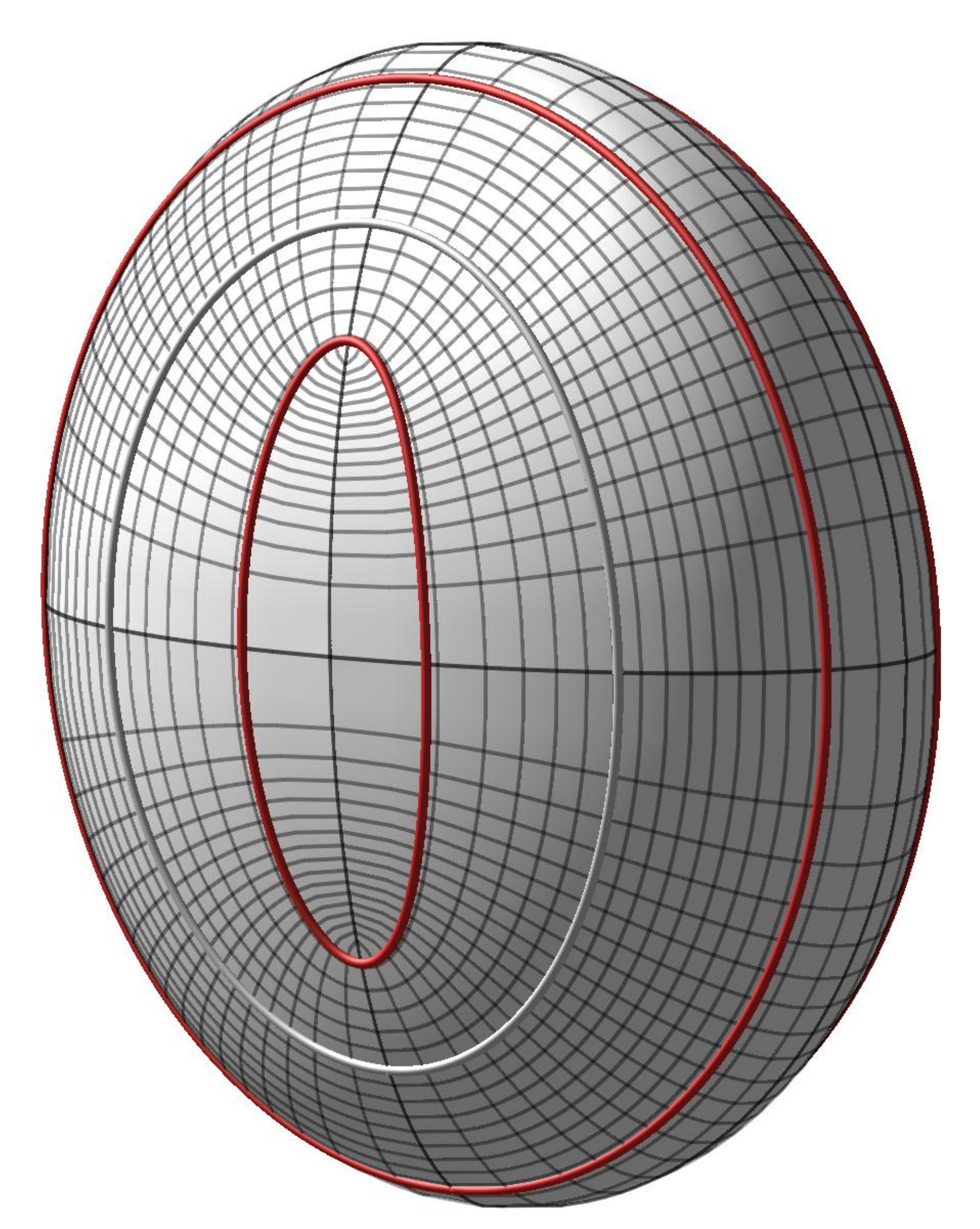}\quad \includegraphics[width=30mm]{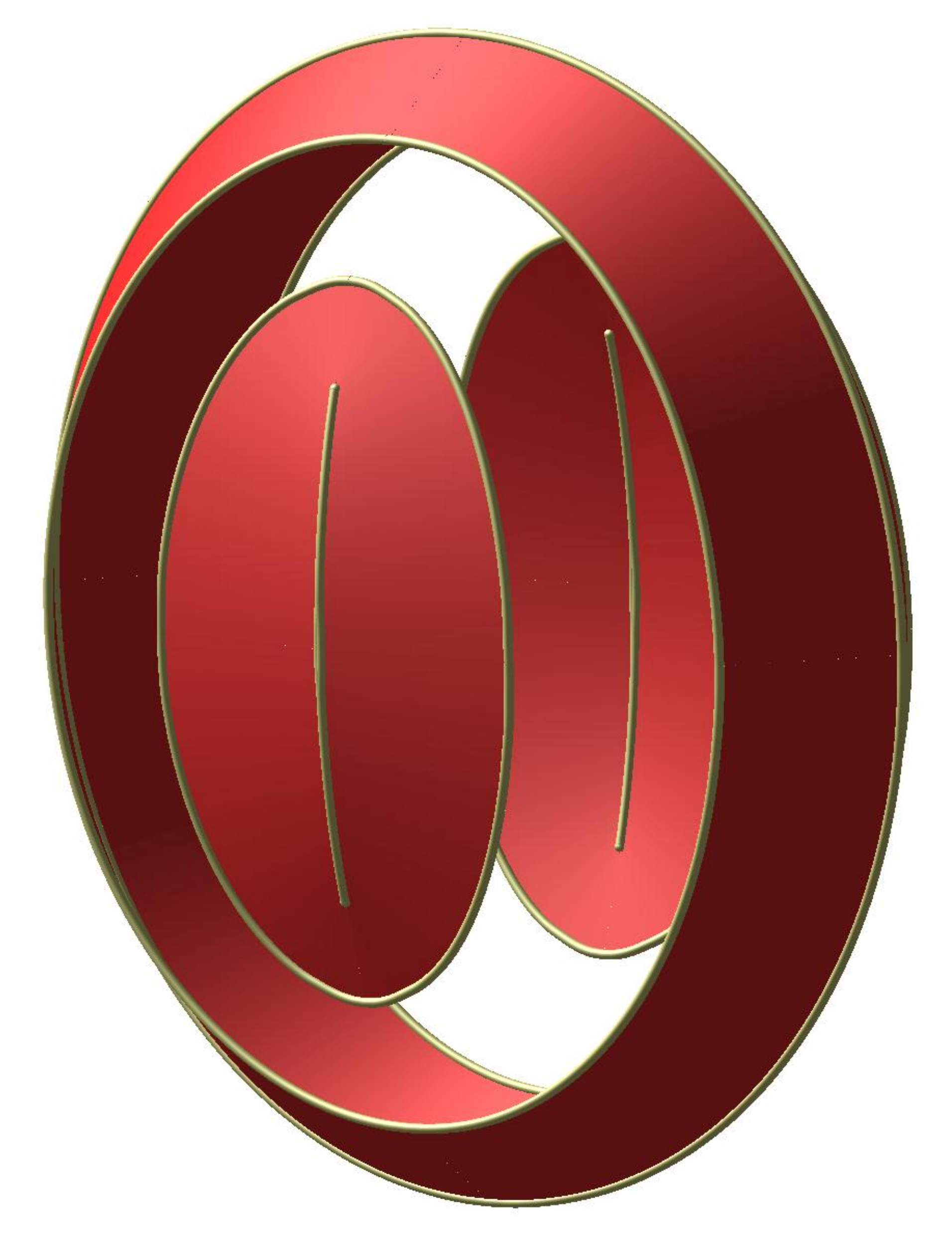}\quad \includegraphics[width=30mm]{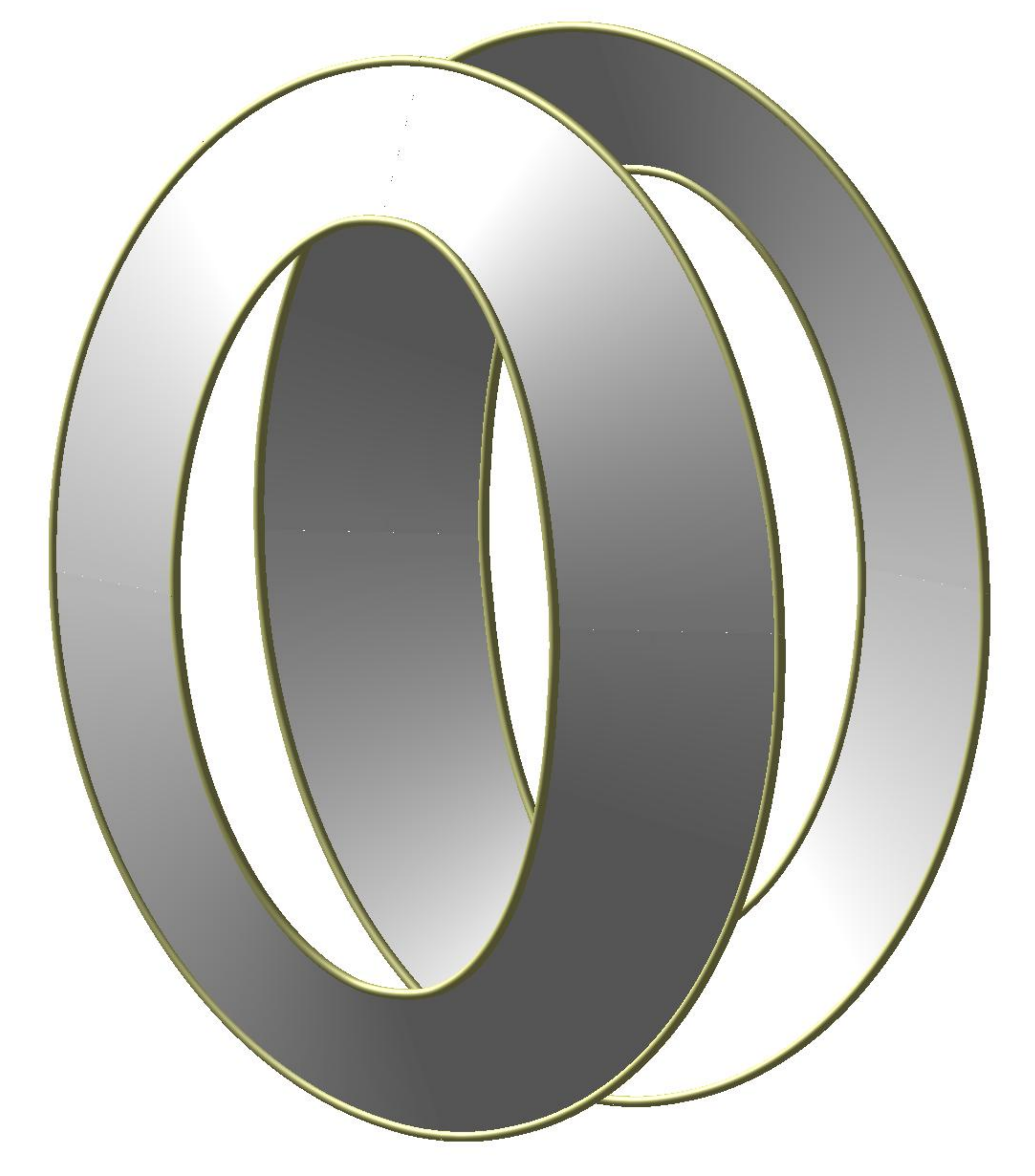}\quad \includegraphics[width=30mm]{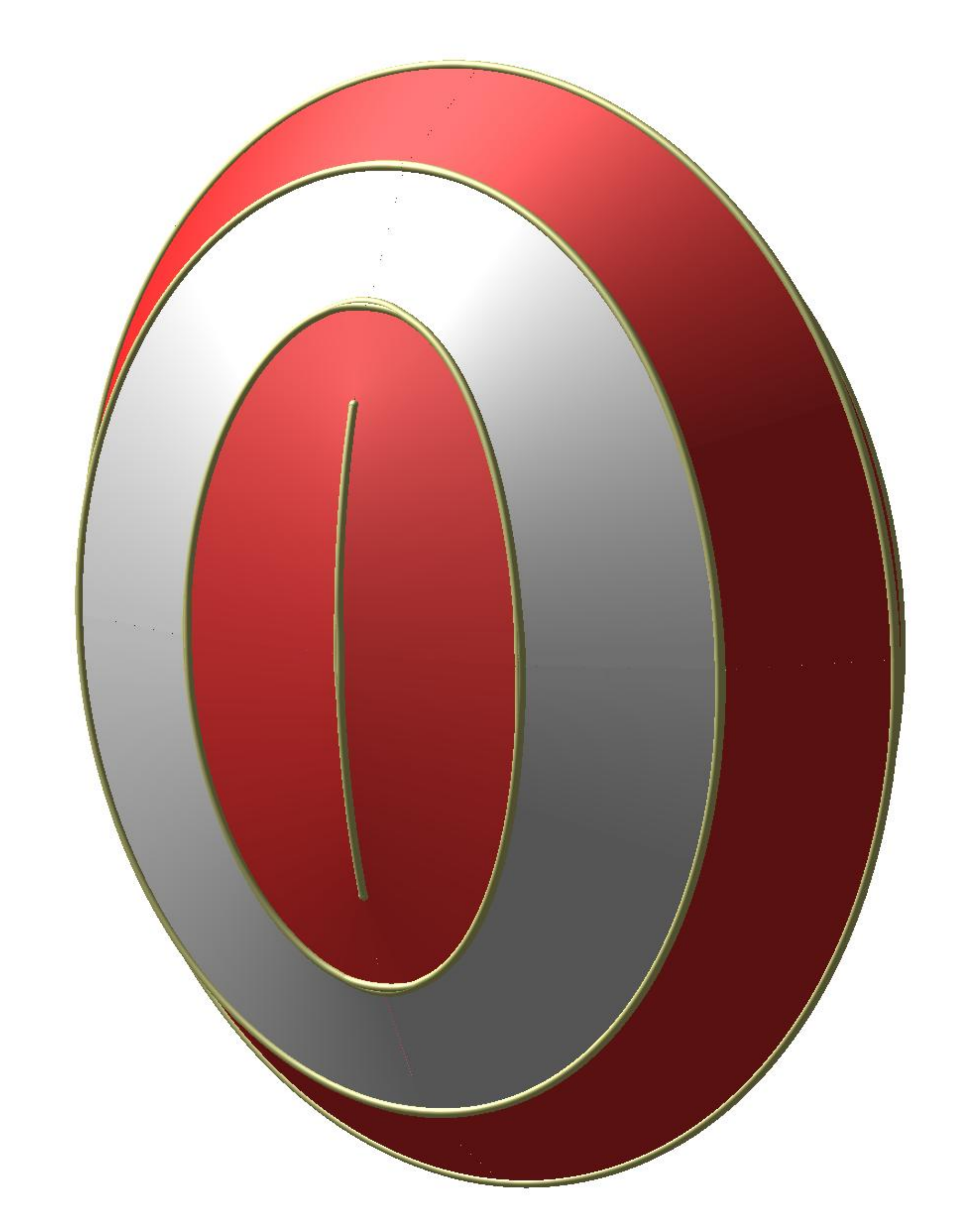}
\caption{Ribbonization of the ellipsoid with half axes $\sqrt{5}$, $2$, and $1$ corresponding to the parameters $(a,b,c) = (5,4,1)$ in the representation \eqref{eqEllipParam}. The ribbonization is built from $6$ Cartan surface ribbons along the indicated line-of-curvature center-lines which intersect the horizontal ``equator'' at equi-distributed points.} \label{figEllipsoidRib1}
\end{center}
\end{figure}

Such curvature line ribbonizations are interesting, partly because they give nontrivial illustrations of the simple measure of goodness established in corollary \ref{corCurvLines}, and partly because they also clearly highlights the
significant umbilical points. The umbilics on the ellipsoid considered here correspond to the four endpoints of the wedge segments
that appear on the top cap and on the bottom cap --  both visible in the second display from the left in figure \ref{figEllipsoidRib1}.

\subsection{Inspection of the ellipsoid}
The ellipsoid has $4$ vertices  -- corresponding to the $4$ umbilical points -- each of degree one, $d_{k} = 1$, and each therefore contributing one-half to the Euler characteristic, see equation \eqref{eqE1}.

\section{Comparison with classical topological inspections}

As illustrated above, the topology of the surface can be read off from a ribbonization -- in fact often in an easier way than from a triangulation. In this section we will briefly compare the above inspection with the methods of Morse and Poincar\'{e}-Hopf based on inspections of Morse height-functions and their corresponding vector fields, respectively.

Consider a Morse height-function $f$ on a surface $S$ and choose center curves for a ribbonization among level curves of $f$. Since the saddle points of $f$ are isolated the center curves can be chosen to be arbitrarily close and yet with tangents avoiding asymptotic directions, so that such ribbonizations exist and have the same topology as the surface. Moreover, as a third perspective, the gradient of $f$ on $S$ represents a vector field whose indices
also count its topology.

Based on a Morse height-function  these three topological inspections are all based on  countings of minima, maxima, and saddle points. Clearly, the final summations give the same result when applying  Table \ref{morse1} below.

\begin{table}[h!]
\centering
\begin{tabular}{|l|r|r|r|r|}
  \hline
  & Ribbon ($d_k$) & Morse ($\gamma$) & Vector field ($I$) \\
  \hline \hline
  Minimum     & 0 & 0 & 1\\
  Saddle point& 4 & 1 & -1\\
  Maximum    & 0      & 2 & 1\\
  \hline
  {\large $\chi$}           & $\frac{1}{2}\sum_{k=1}^{n_v}(2-d_k)$&
  $\sum_{\gamma=0}^{\gamma=2} (-1)^\gamma \, n_\gamma$&
  $\sum_{k=1}^{n_z} I_k$\\
  \hline
\end{tabular}
\vspace{.3cm}
\caption{Listing the relationship between degrees of vertices, $d_{k}$, Morse indices $\gamma$, and the indices, $I$, of a vector field for smooth two-dimensional manifolds. Also compared are the corresponding three topological inspections for the Euler characteristic: the ribbon inspection, based on vertex counting; the Morse index formula, based on critical points of Morse functions \cite{ni2004}; and the Poincar\'{e}-Hopf formula, based on the counting of types of zeros of a vector field \cite{milnor1965}.}
\label{morse1}
\vspace{-0.5cm}
\end{table}

In the case of a torus with its classical Morse height function, see \cite[Diagram 1 p. 1]{milnor1963}, the corresponding ribbonizations all have one minimum, one maximum (both with degree $d_{k} = 0$) and two saddle points (with degrees $d_{k} = 4$), so that the sum is $\frac{1}{2}\sum_{k=1}^{n_v}(2-d_k) = 0$, as expected. An interesting Morse height function for the non-orientable Boy's model of $\mathbb{R}P^{2}$ in $\mathbb{R}^{3}$, that may likewise be used as center curves for a ribbonization, is presented by U. Pinkall in \cite[Chapter 6, pp. 63--67, figures 6.7 and 6.8]{Pinkall1986}. This particular ribbonization has $4$ vertices of degree $d_{k} =0$, and $3$ vertices of degree $d_{k} = 4$, so that $\chi = 1$.

\section{Conclusions}

In this paper we recover the conditions for the existence of proper rollings of one surface on another \cite{nomizu1978, Kobayashi1963, tuncer2007, cui2010} -- in particular the condition that the two contact curves, that are generated from the rolling, have identical geodesic curvature. This follows from defining the standard rollings as rigid motions in $\mathbb{R}^3$ that are conditioned partly via their instantaneous rotation vectors and partly via the obvious condition of contact between the mentioned track curves on the respective surfaces, i.e.~common speed of the contact point along the tracks and common tangent planes at the instantaneous point of contact.

Surfaces are then approximated by a mesh of ribbons. Rolling a surface on a plane and using the Cartan developments of curves allow us to construct developable ribbons that have common tangent planes everywhere along the curve of contact on the surface. In this way we may approximate the surface not just by one such developable surface but by a full set of ribbons. In short, the surface is ribbonized by flat ribbons which have center-curve contact with the surface. This is a clear difference in comparison with the much used method of triangulations, which typically only give discrete point-contact with the surface. In the same way as for triangulations, defining a measure of ``goodness'' of a Cartan ribbonization is dependent on the actual application. Different methods for designing surfaces by developable patches within a desired global error bound have been developed in \cite{liu2007a,liu2009a,tang2016a,rabinovich2018a}. For Cartan ribbonizations, this is an interesting problem, which we have addressed by introducing a local measure of goodness for the approximation of the surface along a single ribbon.

Concerning the global structure of the approximations, we present a particularly simple topological inspection of the ribbonized surfaces, which gives the Euler characteristic of the ribbonization -- and thence also of the surface, if the ribbonization is fine enough.
The ensuing topological formula for the Gauss--Bonnet theorem involves only the vertices of the ribbonization and their degrees.
This complements the classical inspections of topology stemming from Morse theory and from the Poincar\'{e}-Hopf formula, which also amount to summing over critical point indices. If we organise the ribbonization of a given surface according to level curves of a Morse height function, then we obtain the direct correspondence shown in Table \ref{morse1}.

The intriguing relations between the kinematics of rolling and the geometry of developable surfaces clearly carries many more assets for future work than what we cover in the present paper. As indicated above, already the study of ribbonizations could well pave new ways for refined analyses of physical, geometrical, and topological properties of surfaces. Not to mention the potentials of their higher dimensional analogues. Possible practical applications are manifold and appear in such diverse fields as robotics, architecture, design, shape analysis, and modern engineering. See for example the following  works on rolling spherical robots \cite{bai2015a}, roof panelling \cite{mehrtens2007}, statistical geometric regression analysis \cite{fletcher2004a, hinkle2012a}, and the manufacturing of clothes \cite{Rose2007}.
\pagebreak


\end{document}